\documentclass{siamart250211}
\usepackage{graphicx}
\usepackage{epstopdf} 
\usepackage{hyperref}    
\usepackage{subcaption}
\usepackage{amssymb,latexsym,amsmath}
\usepackage{xcolor}
\usepackage{enumerate}
\usepackage{booktabs}
\usepackage{MnSymbol,bbding,pifont}
\usepackage{tikz,tikz-cd}
\usetikzlibrary{hobby}
\usetikzlibrary{arrows}
\usetikzlibrary{positioning}
\usetikzlibrary{shapes,snakes}
\usetikzlibrary{cd}

\DeclareMathOperator{\diag}{diag}

\DeclareMathOperator{\tr}{trace}

\DeclareMathOperator{\skeww}{skew}

\newcommand{\Rnn}{\mathbb{R}^{n\times n}}
\newcommand{\Cnn}{\mathbb{C}^{n\times n}}
\newcommand{\R}{\mathbb{R}}
\newcommand{\C}{\mathbb{C}}

\newcommand{\NN}{\mathcal{N}}
\newcommand{\proj}{\mathrm{Proj}}
\newcommand{\grad}{\mathrm{grad}}
\newcommand{\hess}{\mathrm{Hess}}

\newsiamremark{remark}{{\sc Remark}}
\newtheorem{example}[theorem]{Example}

\newcommand{\hide}[1]{}

\numberwithin{equation}{section}

\begin{document}

\title{Closest normal matrix found again!\\
Using Riemannian optimization\thanks{Submitted to the editors on ???}}

\author{Vanni Noferini\thanks{Aalto University, Department of Mathematics and Systems Analysis, P.O. Box 11100, FI-00076, Aalto, Finland (\email{vanni.noferini@aalto.fi}). Supported by a Research Council of Finland grant (decision number 370932).}
\and Matvei Zhukov\thanks{Corresponding author. Aalto University, Department of Mathematics and Systems Analysis, P.O. Box 11100, FI-00076, Aalto, Finland (\email{matvei.zhukov@aalto.fi}). Supported by a Research Council of Finland grant (decision number 370932).}}

\date{}

\maketitle

\begin{abstract}
We propose an approach based on Riemannian optimization to compute a nearest normal matrix to a given one. The problem can be formulated as the minimization of a smooth function either on the manifold \(U(n)\) of \(n\times n\) unitary matrices or on the flag manifold \(U(n)/U(1)^n\). The flag manifold is particularly suitable for theoretical analysis; we characterize the global maximum of the objective function and prove that, for generic inputs, its local minimizers are finitely many and isolated; in turn, this implies the original nearest normal matrix problem generically has finitely many local minimizers, all with distinct eigenvalues. We also develop a Riemannian trust-region method that improves substantially on classical algorithms and can handle considerably larger matrices, as well as a variant for computing the nearest real normal matrix. The paper is complemented by extensive numerical experiments.
\end{abstract}



\section{Introduction}

A square matrix $N \in \Cnn$ is called \emph{normal} if it commutes with its conjugate transpose, $NN^*=N^*N$. There are many other equivalent definitions \cite{EI,GJSW}, of which the most relevant to the present paper is that $N$ is normal if and only if it is unitarily diagonalizable, i.e., $N=Q \Lambda Q^*$ for a unitary $Q \in U(n) \subsetneq \Cnn$ and a diagonal $\Lambda \in \Cnn$; clearly, $Q$ and $\Lambda$ encode respectively eigenvectors and eigenvalues of $N$. Given a complex square matrix $A$, computing a normal matrix $N$ nearest to it, as well as the distance from $A$ to $N$, is a great classic in numerical linear algebra \cite{Causey,Chu,Gabriel79,Gabriel,GS,Higham,NPR,Ruhe}. Although in principle this problem can be posed with respect to any distance defined on $\Cnn$, in this paper we focus on the Frobenius distance $\|A-N\|_F$. Not only is this  choice the most common in the literature, but it also has the mathematical advantage that the Frobenius distance is induced by a real inner product, and thus a real Hilbert space structure naturally arises on $\Cnn$. Importantly for our article, a further consequence is that every embedded real submanifold of $\Cnn$ inherits a natural Riemannian metric.

Computing a nearest normal matrix has numerous applications. We give just some examples. In numerical linear algebra, the distance from $A$ to the set of normal matrices is related to properties that are important for spectral computations, such as the condition number of the matrix of eigenvectors of $A$ \cite{GV} or the size of the $\varepsilon$-pseudospectra of $A$ \cite{TE}. Classically, bounds on the distance from normality, including for instance Henrici's departure from normality \cite{Higham} (which is an upper bound), have been used to study these properties. Problematically, these bounds sometimes provide extremely inadequate approximations of $\|A-N\|_F$, especially when $\|A-N\|_F/\|A\|_F$ is small \cite{Higham}. Therefore, a numerical method that can compute, or reliably estimate, the actual distance may be preferable. Moreover, if the distance between $A$ and the normal matrix $N$ is certified to be small, then the perfectly conditioned eigenvalues of $N$ can be computed as a reasonable approximation of those of $A$ \cite{NPR}. In quantum mechanics, two observable physical quantities can be measured simultaneously if and only if their corresponding Hermitian matrices $H_1$ and $H_2$ commute \cite{SN}. Suppose that in practice $H_1$ and $H_2$ are only close to being a commuting pair because they are derived from an inaccurate model or from experimental data affected by measurement errors. Then, taking into account that a matrix is normal if and only if its Hermitian and skew-Hermitian parts commute \cite{GJSW}, joint measurability can be restored  by first computing the normal matrix $N$ nearest to $A=H_1 + i H_2$ and next replacing $H_1 \leftarrow (N+N^*)/2$ and  $H_2 \leftarrow (N-N^*)/(2i)$. In control theory, it has been suggested to approximate a transfer function to the normal matrix nearest to it, to improve stability of eigenvalue computation \cite{DK83,DK84}. For a dynamical system modeled by the system of ordinary differential equations $\dot{x}(t) = A x(t)$, the norm of the solution $x(t)$ is guaranteed to converge to $0$ when $A$ is stable, i.e., its eigenvalues have negative real parts. However, for a non-normal $A$, transient growth of the solution is a well documented phenomenon \cite{MV,TE}. If $A$ is close to being normal and for sufficiently small values of $t$, the effect of transient growth can be bounded by observing $\displaystyle \|e^{tA}\| \leq \|e^{tN}\| \cdot  e^{t\|A-N\|}.$

A numerical solution of the nearest normal matrix problem has remained a challenging open problem in numerical linear algebra for some time in the past, until in 1987 it was solved independently by Gabriel \cite{Gabriel} and Ruhe \cite{Ruhe} (and indeed, our article's title is an homage to \cite{Ruhe}). As explained by Higham in the review \cite{Higham}, both Gabriel's and Ruhe's Jacobi-type algorithms are based on earlier work by Causey \cite{Causey}. From the viewpoint of Riemannian optimization, both methods can be seen as variants of coordinate descent on the manifold of unitary matrices. The connection with Riemannian optimization was also exploited by Chu \cite{Chu}, who however did not tackle the general problem described above, but only its variant of computing the nearest real normal matrix with a given spectrum, which is a simpler task. Methods that do not work on a manifold also exist, including some that address structured variants of the problem \cite{GS,NPR}.

The present manuscript revisits this classical problem from the viewpoint of modern optimization on Riemannian manifolds \cite{ABG07,AMS08,Boumal,Manopt}, with both theoretical and algorithmic goals. On the side of theory, we exploit the formulation on the flag manifold \(U(n)/U(1)^n\) to study several properties of the corresponding objective function, which is central to both our proposed algorithm and the classical methods. We characterize its global maximum and we prove that, for generic inputs, its local minimizers are finitely many and nondegenerate; moreover, the corresponding local minimizers of the original nearest normal matrix problem are finitely many and have distinct eigenvalues. We also give a complete analysis for \(n=2\) and show that nonglobal local minima can occur for \(n>2\). On the algorithmic side, replacing the classical Jacobi-type coordinate descent with a Riemannian trust-region method substantially improves practical performance and allows us to tackle considerably larger inputs than those previously documented in the literature \cite{Chu,Gabriel,GS,Ruhe}. We compare implementations on \(U(n)\) and on the flag manifold, study the numerical behavior of different possible initializations, and also develop a variant for the nearest real normal matrix problem, which is of special interest in some applications \cite{GS,Higham}.

While this manuscript was being finalized, Bierly posted the preprint \cite{Kyle} on the normal Procrustes problem, whose specialization to the nearest normal problem partially overlaps with the algorithmic part of the present manuscript. The two articles were developed simultaneously and independently, and \cite{Kyle} cites Zhukov's presentation of this work at FoCM 2026. The scopes are nevertheless different: \cite{Kyle} treats the more general normal Procrustes problem, whereas we focus on the nearest normal problem and develop the flag-manifold formulation and the theoretical analysis of its extrema described above.

The structure of the paper is as follows. In Section \ref{sec:riemann}, we show that the problem of finding a normal matrix $N$ nearest to $A$ is equivalent to minimizing a smooth function $f_A$ on a Riemannian manifold; we also derive the Riemannian gradient and Riemannian Hessian of $f_A$, allowing us to implement second-order optimization algorithms. In Section \ref{sec:theory}, we study some related theoretical questions, expanding the knowledge that was previously available in the literature. Section \ref{sec:real} is devoted to the real analogue of the problem, where $A \in \Rnn$ and a real normal $N$ is sought. In Section \ref{sec:numexp}, we describe the resulting practical algorithm and we illustrate its numerical performance. Section \ref{sec:conclusions} concludes the paper.
\section{Nearest normal matrix as a Riemannian optimization problem}\label{sec:riemann}

Let us start by discussing why and how the problem of computing a nearest normal matrix to $A$ is equivalent to minimizing a function over the manifold of unitary matrices (or equivalently over the flag manifold, see below). This connection has been first made in the pioneering work \cite{Causey} and has later played an important role in the treatment of many other authors \cite{Chu,Gabriel79,Gabriel,Higham,Ruhe}. Besides collecting relevant results obtained by our predecessors, below we will add some new insights. In part, we will also follow the framework developed in \cite{NP}.

Given a matrix $X \in \Cnn$, denote by $\mathcal{N}(X):=X-\diag(X)$ the orthogonal (in the Frobenius inner product $\langle X, Y \rangle = \Re \tr(Y^*X)$) projection of $X$ onto the set of matrices with zero diagonal. In other words, $\mathcal{N}(X)=X-\diag(X)$, where here and throughout $\diag(X)$ denotes the diagonal matrix whose diagonal elements coincide with those of $X$. (In particular, diagonal matrices and matrices with zero diagonal are characterized by, respectively, $\diag(X)=X$ and $\diag(X)=0$.) Recall also that the Frobenius norm $\| A \|_F = \sqrt{\langle A,A \rangle}$ is unitarily invariant, since for every pair of unitary matrices $U,V$ we have by the cyclic property of the trace
\[ \| U A V \|_F^2 = \Re \tr (V^* A^* U^* U A V) = \Re \tr(A^*A) = \| A \|_F^2. \]

Observe now that for any normal matrix the spectral decomposition $X=Q \Lambda Q^*$ holds, where $Q \in U(n)$ is unitary. Hence,
\begin{equation}\label{eq:Riemannmin}
 \min_{X \ \mathrm{normal}} \|A-X\|_F^2 = \min_{Q \in U(n)} \min_{\Lambda \ \mathrm{diagonal} } \| A - Q \Lambda Q^* \|_F^2 = \min_{Q \in U(n)} \| \mathcal{N}(Q^*AQ) \|_F^2.    
\end{equation}
The right hand side of \eqref{eq:Riemannmin} constructs an objective function 
\begin{equation}
\label{def:f}
f_A(Q) := \| \mathcal{S}(Q) \|_F^2, \qquad  \mathcal{S}(Q): = \mathcal{N}(Q^*AQ).  
\end{equation}
 The function $f_A(Q)$ in \eqref{def:f} is defined on the real embedded manifold $U(n) \subsetneq \Cnn \cong \R^{2 n^2}$ of $n \times n$ unitary matrices, and it is in fact a smooth function on $U(n)$. In Proposition \ref{prop:symmetries}, we note some symmetries of $f_A(Q)$.

\begin{proposition}\label{prop:symmetries}
  Given $A \in \Cnn$, let $f_A(Q)=\| \mathcal{S}(Q) \|_F^2$ where $\mathcal{S}$ is defined as above. Then, for every diagonal unitary matrix $D$ and for every permutation matrix $P$, we have $f_A(Q)=f_A(QD)=f_A(QP)$ for all $Q \in U(n)$. 
\end{proposition}
\begin{proof}
 The off-diagonal elements of $D^*(Q^* A Q)D$ and $Q^* A Q$ are equal up to a phase, and hence $f_A$ takes the same value on $Q$ and $QD$. Similarly, $P^T (Q^* A Q)P$ has the same off-diagonal elements as $Q^* A Q$, except that they appear in different positions in the respective matrices.
\end{proof}
The quotient manifold $\mathcal{F}_n:=U(n)/U(1)^n$ is a smooth manifold itself, and it has sometimes been called the \emph{flag manifold}, because it parametrizes the space of the ``complete flags" of $\C^n$, i.e., nested sequences $\{0\} = V_0 \subsetneq V_1 \cdots \subsetneq V_{n-1} \subsetneq V_n = \C^n$ where for $i=0,\dots,n$ the subspace $V_i$ has dimension $i$ over $\C$. Proposition \ref{prop:symmetries} implies that we may as well minimize $f_A(Q)$ on the flag manifold, which has lower dimension than $U(n)$. Although formally an element of $\mathcal{F}_n$ is an equivalence class, one can fix a representative, e.g., a unitary matrix $Q$ whose diagonal elements are nonnegative real numbers.

If a minimizer for $f_A(Q)$ is found on $\mathcal{F}_n$ (or on $U(n)$), it is easy to reconstruct from it a normal matrix nearest to the input $A$. We record this fact in Theorem \ref{prop:equivalent}, which is an analogue for the nearest normal matrix problem of \cite[Theorem 3.2]{NP}, which deals with the nearest stable matrix problem. We note in passing that, due to a subtlety, the local statements in \cite[Theorem 3.2]{NP} are not fully correct as stated; a correction will appear in \cite{NPacta}.

\begin{theorem}\label{prop:equivalent}
Let $A \in \Cnn$ and, for $Q \in U(n)$, consider the normal matrix $N(Q) := Q \diag(Q^* A Q) Q^*$.
\begin{enumerate}
\item $[Q_0]$ is a global minimizer of $f_A$ on $\mathcal{F}_n$ if and only if
$N(Q_0)$ is a global minimizer of $\|A-X\|_F^2$ over the normal matrices.
\item Let $N_0$ be normal and let
$\mathcal{R} := \{ Q \in U(n) : Q^* N_0 Q \text{ is diagonal}\}$.
Then $N_0$ is a local minimizer of $\|A-X\|_F^2$ if and only if, for every
$Q \in \mathcal{R}$, one has $N(Q) = N_0$ and $[Q]$ is a local minimizer of
$f_A$ on $\mathcal{F}_n$.
\end{enumerate}
Moreover, the same correspondences hold for the minimizers of $f_A(Q)$ on $U(n)$.
\end{theorem}

\begin{proof}
By Proposition \ref{prop:symmetries}, it suffices to prove the statements for minimizers on $U(n)$ and the analogues on $\mathcal{F}_n$ follow. If $U \in U(n)$ is diagonal then $\diag(U X U^*)=\diag(X)$ for all $X$. Thus, $N(QU) = N(Q)$, i.e., $N$ only depends on $[Q]$ and is well defined as a function on $\mathcal{F}_n$.
Moreover, for every $Q \in U(n)$ and every diagonal $D \in \Cnn$,
\begin{equation}\label{eq:split}
 \|A - QDQ^*\|_F^2 = \|\mathcal{N}(Q^*AQ)\|_F^2 + \|\diag(Q^*AQ) - D\|_F^2
 \geq f_A(Q),
\end{equation}
with equality if and only if $\diag(Q^*AQ)=D$ if and only if $N(Q)=QDQ^*$.

\begin{enumerate}
\item It suffices to observe that \eqref{eq:split} yields $\|A-X\|_F^2 \geq f_A(Q) \geq \min f_A$ and $\|A - N(Q)\|_F^2 = f_A(Q)$ for every $Q$. Hence the two global minima
coincide and the minimizers correspond.
\item Assume first $N_0$ is a local minimizer, and fix $Q_0 \in \mathcal{R}$ and
$D_0 := Q_0^* N_0 Q_0$. For all $k \geq 1$ construct the sequence
\[ D_k = \frac{k-1}{k} D_0 + \frac1k \diag(Q_0^*AQ_0);   \]
clearly $D_k$ is diagonal, and
$N_k := Q_0 D_k Q_0^*$ is normal with $N_k \to N_0$. Since $N_0$ is a local minimizer, for sufficiently large $k$ it holds $\| A - N_0 \|_F^2 \leq \| A - N_k \|_F^2$. Subtracting $f_A(Q_0)$ from both terms and by
\eqref{eq:split}, this implies 
\[ \|\diag(Q_0^*AQ_0) - D_0\|_F^2 \leq \left(\frac{k-1}{k}\right)^2 \|\diag(Q_0^*AQ_0) - D_0\|_F^2, \]
forcing
$\diag(Q_0^*AQ_0) = D_0$ and hence $N(Q_0) = N_0$. Observe now that the function $N(Q)$ is polynomial in the elements of $Q$, hence continuous. Therefore, whenever $\| Q - Q_0 \|_F^2$ is small enough, $f_A(Q) = \|A-N(Q)\|_F^2 \geq \|A-N_0\|_F^2 = f_A(Q_0)$, proving that $Q_0$ is a local minimizer.

Conversely, assume $N_0$ is not a local minimizer. Then there is a sequence $(N_k)_k$ of normal
matrices with $N_k \to N_0$ and $\|A-N_k\|_F^2 < \|A-N_0\|_F^2$. For all $k$, write
$N_k = Q_k D_k Q_k^*$ with $Q_k \in U(n)$ and $D_k$ diagonal. Since $U(n)$ is
compact, we may assume (up to picking a subsequence of $N_k$) that $Q_k \to Q_0 \in U(n)$. Therefore, 
$D_k = Q_k^* N_k Q_k \to Q_0^* N_0 Q_0$, implying $Q_0 \in \mathcal{R}$. If $N(Q_0) \neq N_0$ there is nothing else to prove. Otherwise,
$f_A(Q_0) = \|A - N_0\|_F^2$, while \eqref{eq:split} yields
\[ f_A(Q_k) \leq \|A - N_k\|_F^2 < \|A-N_0\|_F^2 = f_A(Q_0) \]
for all $k$; hence $Q_0$ is not a local
minimizer of $f_A$ on $U(n)$.
\end{enumerate}
\end{proof}

\begin{remark}
    The other symmetry of $f_A$ discussed in Proposition \ref{prop:symmetries} is less advantageous algorithmically, because taking a quotient by the symmetric group $S_n$ does not reduce dimensionality. We note however that, by a logic similar to that in the proof of Theorem \ref{prop:equivalent}, the change of variable $Q \mapsto QP$, where $P$ is a permutation matrix, also does not change the corresponding normal matrix $N$. 
\end{remark}

\begin{remark}\label{rem:nonder}
    When $N_0$ has distinct eigenvalues and $N_0=N(Q_0),$, it is sufficient that one $Q_0 \in \mathcal{R}$ is a local minimum of $f_A(Q)$ to guarantee that $N_0$ is a local minimum of the distance to $A$. This is generally not true for a derogatory $N_0$. Consider for example $\displaystyle A = \begin{bmatrix}
        0&0&1\\
        0&0&1\\
        -1&-1&\sqrt{7}
    \end{bmatrix}$ and $\displaystyle N_0 = \begin{bmatrix}
        0&0&0\\
        0&0&0\\
        0&0&\sqrt{7}
    \end{bmatrix}$. Using \eqref{eq:fgrad} below, it  is easy to verify that $\displaystyle Q=\begin{bmatrix}
        U & 0\\
        0 & 1
    \end{bmatrix}$ is a stationary point of $f_A(Q)$ for every $U \in U(2)$. We can also compute the Hessian at a tangent direction $Q \Omega$ by \eqref{eq:fhess_Un} below. By parametrizing the skew-Hermitian matrix $\Omega$ one can verify that $Q_0=I$ is a local minimum, but $\displaystyle Q_1=\frac{1}{\sqrt{2}}\begin{bmatrix}
        1&1&0\\
        1&-1&0\\
        0&0&\sqrt{2}
    \end{bmatrix}$ is not because $f_A(Q)$ decreases along the tangent direction $Q_1 \Omega_1, \Omega_1:= \begin{bmatrix}
        0&0&i\\
        0&0&0\\
        i&0&0
    \end{bmatrix}$. Hence
$N(Q_1e^{t\Omega_1})\to N_0$ are normal matrices closer to $A$ than $N_0$ and
$N_0$ is not a local minimizer.
\end{remark}

Theorem \ref{prop:equivalent} yields the main idea of this paper: We can minimize $f_A(Q)$ on a Riemannian manifold  by using Riemannian optimization algorithms \cite{ABG07,AMS08,Boumal}, and this is tantamount to finding a normal matrix nearest to the input matrix $A$. Note that the manifold of choice can be either $U(n)$ or $\mathcal{F}_n$. While $n^2=\dim_\R U(n) > \dim_\R \mathcal{F}_n = n^2-n$, optimizing on the flag manifold does not necessarily yield a computational advantage. In practice, this depends on the balance between the reduced dimension and the overhead caused by explicit computations of additional projections. See Section \ref{sec:implement} for a more detailed analysis and a numerical comparison.

\begin{remark}
Recently, other algorithms for matrix nearness problems based on optimization over unitary matrices have been studied \cite{DNN,NN,NP}.  The objective function for the complex case in \cite{NP} has the same phase symmetries as \eqref{def:f}, and hence the corresponding optimization task could also be mathematically formulated on $\mathcal{F}_n$. Similar comments also apply to \cite{DNN,NN}, with the additional caveat that the objective functions in those papers are not everywhere differentiable, which could in principle result in a different relative numerical behavior on $U(n)$ and $\mathcal{F}_n$.
\end{remark}

For the nearest normal matrix problem, the approach to minimize a function of a unitary matrix on $U(n)$ was already proposed, either explicitly or implicitly, in \cite{Chu,Gabriel,Ruhe}. However, \cite{Gabriel,Ruhe} apply an ad hoc Jacobi method, which can be interpreted as a peculiar kind of coordinate descent on $U(n)$, where each ``coordinate" corresponds to  those Givens rotations that act on a fixed pair of rows or columns. Perhaps unsurprisingly, the resulting algorithms inherit the known limitations and potential numerical pitfalls of coordinate descent methods. On the other hand, the focus on \cite{Chu} is not exactly on the nearest normal matrix problem, but rather on the related (and easier) problem of finding a nearest real normal matrix with prescribed spectrum; this implies that many of the results derived there are not directly applicable to the goals of our paper. In addition, \cite{Chu,Gabriel,Ruhe} were all written in an era preceding the most recent developments on both  theory \cite{ABG07,AMS08, Boumal} and computation \cite{Manopt} of modern optimization on manifolds. To summarize, even if the basic idea can certainly be traced back to the ample literature on this problem, we believe that, in the context of the nearest normal matrix problem, the tools of Riemannian optimization were previously not fully exploited. 

In the next subsections, we derive the gradient and Hessian of $f_A(Q)$, which are crucial ingredients for any  Riemannian optimization algorithm.

\subsection{Euclidean and Riemannian gradients}
It is convenient to start with Lemma \ref{lm:costgradient}, which is more general than what we need, and it recovers for instance \cite[Subsection 6.2]{NP} as a special case. Here and below recall that, for every pair $X,Y \in \Cnn$, their commutator is denoted by $[X,Y]:=XY-YX$.
\begin{lemma}
     \label{lm:costgradient}
         Let $A \in \Cnn$, let $h\colon U(n)\to \R$ be a function of the form $h(Q) = \|\mathcal{P}(Q^*AQ)\|^2_F$, where $\mathcal{P}$ is a linear orthogonal (in the Frobenius inner product) projection $\mathcal{P}$ in a neighborhood of $Q\in U(n).$ Then, the Euclidean gradient of $h$ at $Q$ is 
         \begin{equation}\label{eq:Egrad}
             \nabla h(Q) = 2(A^*QS + AQS^*),
         \end{equation}
         and its Riemannian gradient at $Q$ is
         \begin{equation}\label{eq:Rgrad}
             \grad h(Q) = 2Q\skeww([B^*, S]),
         \end{equation}
         where $B = Q^*AQ,\, S = \mathcal{P}(B),$ and $\skeww(X) = \frac{1}{2}(X - X^*).$

     \end{lemma}
    \begin{proof}
    Using that $\mathcal{P}$ is linear and that (being an orthogonal projection) it satisfies $\mathcal{P}^2=\mathcal{P}$ and $\langle \mathcal{P}(X),Y \rangle = \langle X, \mathcal{P}(Y) \rangle$, we have
    \begin{equation*}
    \begin{split}
    dh &= d\tr(\mathcal{P}(B)\mathcal{P}(B)^*)= \tr(d\mathcal{P}(B)S^* + Sd\mathcal{P}(B)^*) = 2\Re\tr(S^*dB) \\
    &= 2\Re\tr(S^*(dQ^*AQ + Q^*AdQ))= 2\Re\tr((AQS^* + A^*QS)dQ^*).
    \end{split}
    \end{equation*}

    Thus, the Euclidean gradient of $h$ equals to 
    $$
    \nabla h(Q) = 2(A^*QS + AQS^*) = 2 Q(B^*S + BS^*).
    $$

    From that, the Riemannian gradient can be calculated by projecting $\nabla h(Q)$ onto the tangent space $\mathcal{T}_Q U(n)$.
    \begin{equation*}
         \begin{split}
        \grad h(Q) &= \proj_{\mathcal{T}_Q U(n)}\nabla h(Q) = Q\skeww(Q^*\nabla h(Q)) \\
        &= \frac{1}{2}Q(Q^*\nabla h(Q) - \nabla h(Q)^*Q) = A^*QS + AQS^* - QS^*Q^*AQ - QSQ^*A^*Q \\
        &= Q(B^*S + BS^* - S^*B - SB^*)= 2Q\skeww(B^*S - SB^*) = 2Q\skeww([B^*, S]).
    \end{split} 
    \end{equation*}
    \end{proof}

\begin{corollary}\label{cor:Egradf}
    Let $f_A : U(n) \rightarrow \R$ be defined as in \eqref{def:f}. Then 
    \begin{equation}
        \label{eq:fgrad}
        \nabla f_A(Q) = 2 Q(B^*S + BS^*), \qquad \grad f_A(Q) = 2Q\skeww([B^*, S]),
    \end{equation}
    for every $Q\in U(n)$ and having defined $B=Q^*A Q$ and $S=\NN(B)$.
\end{corollary}
\begin{proof}
    Since $f$ satisfies the assumptions of Lemma \ref{lm:costgradient}, the statement follows directly by \eqref{eq:Egrad} and \eqref{eq:Rgrad}.
\end{proof}

\begin{remark}\label{rem:gradflag}
  By Proposition \ref{prop:symmetries},  $f(Q)$ is constant on the orbits of $U(1)^n$. Hence, by the general theory of optimization on manifolds \cite{AMS08, Boumal}, its Riemannian gradient on the flag manifold $\mathcal{F}_n$ coincides with its Riemannian gradient on $U(n)$. We conclude that Corollary \ref{cor:Egradf} is also applicable to algorithms that optimize $f(Q)$ on $\mathcal{F}_n$. A more ad hoc argument that leads to the same conclusion consists of checking that $\diag(\skeww([B^*,S]))=0$, which implies that $\grad f_A(Q) \in \mathcal{T}_{[Q]}\mathcal{F}_n$.
\end{remark}

\begin{remark}\label{rem:stationary}
    From Corollary \ref{cor:Egradf}, we see that $Q$ is a stationary point of $f$ if and only if $[B^*,\diag(B)]$ is Hermitian where $B=Q^*AQ$. In particular \[  B_{ij}\overline{(B_{ii} - B_{jj})} = \overline{B_{ji}}(B_{jj} - B_{ii})\]
    holds for all $i,j$. When $\diag(B)$ has distinct diagonal elements, this implies in turn $|B_{ij}|=|B_{ji}|$.
\end{remark}

\subsection{Euclidean and Riemannian Hessians}

We now turn to computing the Hessian of $f_A$. Let us start with the Euclidean version of the second derivatives.

\begin{proposition}
    Let $f_A$ be defined as in \eqref{def:f}. Its Euclidean Hessian at $Q \in U(n)$, applied to the direction $Z \in T_QU(n)$ is
    \[ \nabla^2 f_A(Q)[Z] = 2A^*(ZS+Q\NN(dB[Z]))+2A(ZS^*+Q\NN(dB[Z]^*)),  \]
    where $S=\NN(Q^*AQ)$ and $dB[Z]:=Q^*AZ+Z^*AQ$.
\end{proposition}
\begin{proof}
 Differentiating $\nabla f_A$ in \eqref{eq:fgrad},
    \begin{equation*}
         \begin{split}
    \frac{1}{2}\nabla^2 f_A(Q)[Z] &= AZS^* + AQdS[Z]^* + A^*ZS + A^*QdS[Z]\\
    &= AZS^* + AQ\NN(dB[Z])^* + A^*ZS + A^*Q\NN(dB[Z]).
    \end{split}
    \end{equation*}
\end{proof}

To obtain the Riemannian Hessian, let us first recall a general result on the covariant derivative of a vector field on a Riemannian manifold \cite{AMS08}.

\begin{lemma}[{\cite[Section 5.3.3]{AMS08}}]
    \label{lm:rgrad2rhess}
    Let $\mathcal{M}$ be a Riemannian submanifold of $\R^N$. Let $h\colon \mathcal{M} \to \R$ be a smooth function and let $G$ be a smooth vector field on $\R^N$ that extends $\grad h$ in a neighborhood of $X \in \mathcal{M}.$ Then, the Riemannian Hessian of $h$ at $X$ in the direction $Z\in \mathcal{T}_X\mathcal{M}$ is given by
    \begin{equation*}
        \hess h(X)[Z] = \proj_{\mathcal{T}_X\mathcal{M}}(\delta G(X)[Z]),
    \end{equation*}
    where $\delta G(X)[Z]$ denotes the standard Euclidean directional derivative of $G$ at $X$ in the direction $Z.$
\end{lemma}

Unlike what happens for the Riemannian gradient, the Riemannian Hessian depends on whether we view $f_A$ as a function on the domain $U(n)$ or on the domain $\mathcal{F}_n$. In principle, our algorithm can be implemented on either manifold (the former is somewhat easier to implement numerically, while the latter is lower dimensional leading to potential algorithmic advantages); hence, for completeness we treat both cases. Note that \eqref{eq:fhess_Un} and \eqref{eq:fhess_quotient} appear similar, but differ by the application of the operator $\NN$ to the outer occurrence of the operator $\skeww$. To avoid an excessively cumbersome notation, in Corollary \ref{cor:Rhessian} we borrow from the Riemannian optimization community \cite[Chapter 5]{AMS08} the convention to identify the abstract tangent space at a given equivalence class $[Q]$ with the ``horizontal space" at a chosen representative matrix $Q$. An abstract vector tangent to the flag manifold is thus identified with a concrete vector.

\begin{corollary}
        \label{cor:Rhessian}
Let $f_A$ be defined as in \eqref{def:f}, $Q \in U(n)$, $B=Q^*AQ$, and $S=\NN(B)$. For every $Z = Q\Omega \in \mathcal{T}_Q U(n)$, where $\Omega = -\Omega^*$, the Riemannian Hessian of $f_A$ on $U(n)$ is
        \begin{equation}
            \label{eq:fhess_Un}
            \hess_{U(n)} f_A(Q)[Q\Omega] = 2Q\skeww\left(\Omega\skeww([B^*, S]) + [[B, \Omega]^*, S] + [B^*, \NN([B, \Omega])]\right).
        \end{equation}
Furthermore, let $\mathcal{F}_n = U(n)/U(1)^n$ denote the flag manifold. Fixing a representative $Q \in U(n)$ of an equivalence class in $\mathcal{F}_n$, and for every tangent vector $Z = Q\Omega \in \mathcal{T}_{[Q]} \mathcal{F}_n$, where $\Omega = -\Omega^*$ and $\diag(\Omega) = 0$, the Riemannian Hessian of $f_A$ on $\mathcal{F}_n$ is
        \begin{equation}
            \label{eq:fhess_quotient}
            \hess_{\mathcal{F}_n} f_A(Q)[Q\Omega] = 2Q\NN\left(\skeww\left(\Omega\skeww([B^*, S]) + [[B, \Omega]^*, S] + [B^*, \NN([B, \Omega])]\right)\right).
        \end{equation}
    \end{corollary}
    \begin{proof}
    To derive \eqref{eq:fhess_Un}, we specialize Lemma \ref{lm:rgrad2rhess} as appropriate. Let $G(Q) = 2Q\skeww([B^*, S])$ be the smooth extension of $\grad f_A(Q)$. Then,
    \begin{equation*}
        \delta G(Q)[Q\Omega] = 2(Q\Omega)\skeww([B^*, S]) + 2Q\skeww(\delta[B^*, S]) .
    \end{equation*}
    Observe that $\delta[B^*, S] = [\delta B^*, S] + [B^*, \delta S]$. Moreover, $\delta B = \delta(Q^*AQ) = -\Omega B + B\Omega = [B, \Omega]$, and hence $\delta B^* = [B, \Omega]^*$. Therefore, by linearity of $\NN$, $\delta S = \NN(\delta B) = \NN([B, \Omega])$. Substituting into $\delta G$ yields
    \begin{equation*}
        \delta G(Q)[Q\Omega] = 2Q\Omega\skeww([B^*, S]) + 2Q\skeww([[B, \Omega]^*, S] + [B^*, \NN([B, \Omega])]).
    \end{equation*}
    The statement follows by projecting $\delta G(Q)[Q\Omega]$ onto $\mathcal{T}_Q U(n)$. Indeed, the orthogonal projection is $\proj_{\mathcal{T}_Q U(n)}(X) = Q\skeww(Q^*X)$, yielding \eqref{eq:fhess_Un}.

    To derive \eqref{eq:fhess_quotient}, we use the general property of quotient manifolds \cite[Section 5.3.4]{AMS08} that the Riemannian Hessian on $\mathcal{F}_n$ is the orthogonal projection of \eqref{eq:fhess_Un} onto the horizontal space $\mathcal{H}_Q := \{ Q\Omega \in \mathcal{T}_Q U(n) \mid \diag(\Omega) = 0 \}$. Hence, the orthogonal projection onto $\mathcal{H}_Q$ of a vector $V$ is $Q\NN(Q^*V)$. Applying this projection to \eqref{eq:fhess_Un}, we obtain \eqref{eq:fhess_quotient}.
    \end{proof}

\section{Some theoretical results on the function $f_A(Q)$}\label{sec:theory}

A first basic question on the problem of finding a nearest normal matrix is with how much confidence we can speak of ``the" nearest normal matrix; this is easily settled.

\begin{theorem}
    For almost every matrix $A \in \Cnn$, there is a unique normal matrix $N$ nearest to $A$ in the Frobenius distance.
\end{theorem}
\begin{proof}
    The set of normal matrices is evidently a closed set. It was proved by Erd\H{o}s \cite{Erdos} that the medial axis of any closed set has Lebesgue measure zero.
\end{proof}

We now turn to the equivalent formulation of the problem as a minimization task over the manifold $\mathcal{F}_n$. Theorem \ref{prop:max} argues that, interestingly (and in contrast with its minimum), its maximum can be found analytically and is generically achieved at quite many points.

\begin{theorem}\label{prop:max}
    The global maximum of the function $f_A$ in \eqref{def:f} is equal to \[ \max f_A = \| A \|_F^2 - \frac1n |\tr(A)|^2.\] Moreover, if $n \geq 3$  then, for almost every $A$, this maximum value is attained at uncountably many $[Q] \in \mathcal{F}_n$.
\end{theorem}
\begin{proof}
Let $a:=\frac{\tr(A)}n$. For the first part of the statement, noting the identities 
\[f_A([Q])=f_{A-aI}([Q]) \ \text{and} \  \|A-aI\|_F^2 = \|A\|_F^2-\frac1n |\tr(A)|^2,  \]
we may with no loss of generality assume $\tr(A)=0$. Because the diagonal and off-diagonal part of a matrix are orthogonal (in the Frobenius inner product) subspaces of $\C^{n \times n}$, we have that $f_A([Q]) = \| A \|_F^2 - \| \diag(Q^*AQ)\|_F^2 \leq \|A\|_F^2$, and the bound is attainable because every traceless matrix is unitarily similar to a matrix with zero diagonal \cite[Corollary 1]{Fillmore}.

We may restrict to traceless matrices also for the second part. Indeed, if a subset $\mathcal{S}$ has Lebesgue measure zero in $ \{ A : \tr(A)=0 \}$, then the subset $\mathcal{S} \times \C$ also has Lebesgue measure zero in $\{ A : \tr(A)=0 \} \times \C \cong \Cnn$. Consider now the smooth map 
$F([Q],A)=\diag(Q^*AQ)$, viewed as having domain $\mathcal{F}_n \times \{ A:\tr(A)=0\}$ and image $\{  D : D=\diag(D), \tr(D)=0\}$.
Since $F$ is a submersion, $F^{-1}(0)$ is a smooth submanifold of $\mathcal{F}_n \times \{ A:\tr(A)=0\}$. Moreover, by Sard's Theorem \cite[Theorem 6.10]{Lee}, almost every $A$ is a regular value of the projection \[ \pi : F^{-1}(0) \rightarrow \{ A:\tr(A)=0\}, \qquad \pi(([Q],A)) := A.\]

Fix such a regular value $A$; taking into account also that the differential of $F$ with respect to $A$ is surjective, then $0$ is a regular value of
the map
\[
\Phi_A:\mathcal F_n\rightarrow
\{D\in\Cnn:D=\diag(D),\ \tr(D)=0\},
\ 
\Phi_A([Q])=\diag(Q^*AQ).
\] Again by the Regular Value Theorem \cite[Corollary 5.14]{Lee}, we conclude that, for almost every traceless matrix $A$, the set  $\{ [Q] \in \mathcal{F}_n : \diag(Q^* A Q)=0 \}$
is a smooth submanifold of $\mathcal{F}_n$. By a simple computation, its real dimension is
\[ \dim_\R \mathcal{F}_n - \dim_\R \{  D : D=\diag(D), \tr(D)=0\} = (n^2 - n) - (2n-2 ) = (n-1)(n-2),   \]
which is positive for $n \geq 3$.
\end{proof}
Note that the wild behaviour of the global maximizer described by  Theorem \ref{prop:max} is in some sense artificial; indeed, every global maximizer corresponds, through the map $N(Q)$, to the same normal matrix
$(\tr(A)/n)I$. 

In Theorem \ref{thm:sard}, we now study the local minimizers of $f_A$ which, fortunately for our algorithmic developments, are much better behaved than its global maximizers. We start by the auxiliary Lemma \ref{lem:gendiagdist}; note that, in particular, its assumptions are satisfied by every local minimizer $[Q] \in \mathcal{F}_n$.

\begin{lemma}\label{lem:gendiagdist}
  For almost every $A \in \Cnn$, if $[Q] \in \mathcal{F}_n$ is a stationary point of $f_A$ in \eqref{def:f} such that $\hess_{\mathcal{F}_n} f_A(Q)$ is positive semidefinite, then the matrix $B=Q^*AQ$ has distinct diagonal elements.
\end{lemma}

\begin{proof}
    Let $[Q]$ have the  assumed properties. Up to relabelling some indices, suppose $B_{11}=B_{22}$ and parametrize the skew-Hermitian matrix $\Omega$ as $\Omega_{12}=-\overline{\Omega_{21}}= z \neq 0$ and $\Omega_{ij}=0$ for all other values of $(i,j)$. From \eqref{eq:fhess_quotient} and noting that $\skeww([B^*,S])=0$ because $[Q]$ is stationary, for all $z$ it holds
    \[  0 \leq \langle Q \Omega, \hess_{\mathcal{F}_n} f_A(Q)[Q\Omega] \rangle  = - 2 \| \diag([B,\Omega]) \|_F^2 = -4 |B_{12} \overline{z} + B_{21} z|^2,   \]
    implying $B_{12}=B_{21}=0$. 

    Now let
$\lambda_1,\ldots,\lambda_r$ denote the distinct diagonal elements of $B$, with multiplicities
$k_1,\ldots,k_r$, where $r<n$. Up to a permutation that groups together equal diagonal entries, the reasoning above shows that
the diagonal blocks of $B$ have the form $\lambda_a I_{k_a}$. Moreover, the same argument as in Remark \ref{rem:stationary} shows that the off-diagonal blocks below the block diagonal are determined by those above the block diagonal. Given a multiplicity pattern, the matrices $A$ allowing repeated diagonal elements of $B$ admit a parametrization whose degrees of freedom are those of the decomposition $\C^n = \bigoplus_a V_a$ (with $\dim V_a = k_a$), the diagonal elements $\lambda_a$, and the blocks of $B$ above the diagonal. The real dimension of this parameterization is
\[ \underbrace{n^2 - \sum_a k_a^2}_{\dim_\R U(n)/\prod_a U(k_a)} + \underbrace{2r}_{\text{diagonal} \ \text{elements}} + \underbrace{n^2 - \sum_a k_a^2}_{\text{off-diagonal} \ \text{blocks}}. \]
Thus, for any fixed multiplicity pattern, the codimension is $2 \sum_a (k_a^2-1) \geq 6$. 
There are finitely many possible patterns, and hence \cite[Theorem 6.9]{Lee} implies the statement.
\end{proof}

\begin{theorem}\label{thm:sard}
There exists a set $\mathcal J\subset\Cnn$ of Lebesgue measure zero
such that, for every $A\notin\mathcal J$, the function $f_A$ in
\eqref{def:f} has finitely many local minimizers on $\mathcal F_n$.
Moreover, for every $A \not \in \mathcal J$ and for every local minimizer $[Q]$, the diagonal entries of $Q^*AQ$ are
pairwise distinct and the Riemannian Hessian of $f_A$ is positive
definite.
\end{theorem}

\begin{proof}
Let $\mathcal{J}_1 \subset \Cnn$ denote the set of matrices that are exceptional in the sense of Lemma \ref{lem:gendiagdist}. In other words, for every $A\notin\mathcal J_1$,
every stationary point of $f_A$ with positive semidefinite Hessian has
pairwise distinct diagonal entries.

 Consider the open set
\[
 \mathcal U
 :=
 \left\{
 (Q,A)\in U(n)\times\Cnn:
 (Q^*AQ)_{ii}\neq (Q^*AQ)_{jj}\quad\text{for }i\neq j
 \right\}
\]
and define
\[
 F:\mathcal U\longrightarrow\mathcal T_I\mathcal F_n,
 \qquad
 F(Q,A):=Q^*\grad f_A(Q).
\]
For a fixed $A$, the zeros of $F$ are precisely the stationary points of $f_A$ on $U(n)$. Denoting $B:=Q^*AQ$, we see that $\partial_BF(I,B)[dB]$ is equal to
\[ [\diag(dB),B^*]+[\diag(B),dB^*]+ [\diag(dB^*),B]+[\diag(B^*),dB].
\]
Restrict now to $\diag(dB)=0$ and write $dB=H+K$, where
$H=H^*$ and $K=-K^*$. Then
\[
 \frac12\partial_BF(I,B)[dB]
 =
 \sqrt{-1} \cdot [K,\Im\diag(B)]-[H,\Re\diag(B)].
\]
For $i>j$, its $(i,j)$ entry is
\begin{equation}\label{eq:surj}
    H_{ij}\bigl(\Re B_{ii}-\Re B_{jj}\bigr)
 +
 \sqrt{-1} \cdot K_{ij}\bigl(\Im B_{jj}-\Im B_{ii}\bigr). 
\end{equation}
Since $B_{ii}\neq B_{jj}$ by construction, at least one of the coefficients between brackets in \eqref{eq:surj}
is nonzero. Moreover, $H_{ij}$ and $K_{ij}$ can be chosen
independently, and hence \eqref{eq:surj}  can take any value in $\C$.
It follows that $\partial_BF(I,B)$, and hence (via Corollary \ref{cor:Egradf} and Remark \ref{rem:stationary}) $\partial_AF(Q,A)$,
is surjective. 

Therefore, $0$ is a regular value of $F$. By the Regular Value
Theorem \cite[Corollary 5.14]{Lee},
$ \mathcal M:=F^{-1}(0)$
is a smooth submanifold of $\mathcal U$ having real dimension $ \dim_{\R}\mathcal M
 =
 (n^2+2n^2)-(n^2-n)
 =
 2n^2+n$. On the other hand, consider the projection
\[
 \pi:\mathcal M\longrightarrow\Cnn,
 \qquad
 \pi(Q,A)=A.
\]
By Sard's Theorem \cite[Theorem 6.10]{Lee}, the set $\mathcal J_2$ of critical values of $\pi$
 has Lebesgue measure zero. Define $\mathcal J = \mathcal{J}_1 \cup \mathcal{J}_2$; clearly $\mathcal J$ has Lebesgue measure zero.

Fix now $A\notin\mathcal J$, and let $[Q]$ be a local minimum of $f_A$ so that by Lemma \ref{lem:gendiagdist} $Q^* A Q$ has distinct diagonal elements. Since $A$
is a regular value of $\pi$, the differential of $\pi$ is surjective
at $(Q,A)$, and hence $
 \dim\ker d\pi_{(Q,A)}
 =
 \dim\mathcal M-\dim_{\R}\Cnn
 =
 n$. In particular, one can verify that $\ker d \pi_{(Q,A)} = \{ (Q \Omega, 0) : \Omega = -\Omega^* = \diag(\Omega) \}$.

 Suppose for a contradiction that the Hessian of $f_A$ on $\mathcal F_n$ is singular at
$[Q]$. Then, there exists a nonzero skew-Hermitian matrix $\Omega$
with $\diag(\Omega)=0$ such that $\hess_{\mathcal F_n}f_A([Q])[Q\Omega]=0$. Differentiating
$F(Q,A)=Q^*\grad f_A(Q)$ in the direction $Q\Omega$ and using
Corollary~\ref{cor:Rhessian},
\[
 \partial_QF(Q,A)[Q\Omega]
 =
 Q^*\hess_{\mathcal F_n}f_A([Q])[Q\Omega]
 =
 0.
\]
Hence $(Q\Omega,0)\in\ker d\pi_{([Q],A)}$ and therefore $\Omega=\diag(\Omega)=0$. We conclude that the Hessian is nonsingular and hence positive definite. 

Finally, suppose for a contradiction that $f_A$ has infinitely many
distinct local minimizers. Let $[Q_k]$ be a sequence of minimizers.
Since $\mathcal{F}_n$ is compact, after taking a subsequence we may assume
that $[Q_k]\to [Q]$ for some $[Q] \in \mathcal{F}_n$. By \eqref{eq:fgrad} and
\eqref{eq:fhess_quotient}, $[Q]$ is stationary and its Hessian is
positive semidefinite; moreover, since $A\notin\mathcal J_1$, the diagonal
entries of $Q^*AQ$ are pairwise distinct. Therefore, repeating the argument above, the Hessian at $[Q]$ is positive definite, which is a contradiction because by construction $[Q]$ is not an
isolated stationary point.

\end{proof}

\begin{remark}
    Of course, since $f_A$ is  continuous on the compact manifold $\mathcal{F}_n$, it has at least one local (and global) minimizer. In addition, recall from Proposition \ref{prop:symmetries} that $f_A(Q)=f_A(QP)$ for every permutation matrix $P \in S_n$. Since the equation $QD=QP$, with $Q\in U(n),D \in U(1)^n, P \in S_n$, implies $P=D$ and hence $P=D=I$, there are in fact at least $(n!)$ minimizers of $f_A$ on $\mathcal{F}_n$.  
\end{remark}

Theorem \ref{thm:sard} has some interesting implications. First, to obtain a finiteness result on the set of local minimizers, it is unavoidable to quotient by the diagonal unitary matrices, i.e., consider the optimization task on $\mathcal{F}_n$. Indeed, on $U(n)$, each of the stationary points on the flag manifold is necessarily associated to uncountably many unitary matrices, obtained by right multiplication by an arbitrary unitary diagonal matrix. (Recall from Proposition \ref{prop:symmetries} that every unitary matrix belonging to the same equivalence class in the flag manifold is associated to the same normal matrix, so this abundance is in a sense artificial; but it may nevertheless cause numerical issues.) Excluding a null set of ``bad inputs" is also necessary. For example, if $A$ is a multiple of the identity matrix, then $f_A(Q)=0$ is constant and hence every $Q \in \mathcal{F}_n$ is a local minimizer. We conclude this subsection by deducing a corollary of Theorem \ref{thm:sard}.

\begin{corollary}
  For almost every $A\in\Cnn$, over the set of normal matrices
there are finitely many local minimizers of the function $\| A - X \|_F^2$. Moreover, every such minimizer has distinct eigenvalues, and corresponds to precisely $n!$ local minimizers of $f_A$ on
$\mathcal F_n$.
\end{corollary}

\begin{proof}
The statement follows by combining Theorem~\ref{prop:equivalent}, Remark~\ref{rem:nonder}, and Theorem~\ref{thm:sard}.
\end{proof}

\subsection{A full analytic solution for $n=2$}

In $\C^{2 \times 2}$, the nearest normal matrix problem can be solved analytically \cite{Causey,Gabriel,Higham}. Here, we take a different viewpoint by reinterpreting this classic result in terms of the equivalent problem of minimizing \eqref{def:f}. In other words, we go beyond the (known) analytic description of the normal matrices nearest to $A \in \C^{2 \times 2}$ and also provide a novel (to our knowledge) analysis of the associated equivalent classes of optimal unitary matrices in $U(2)/U(1)^2$. In both Example \ref{ex:1} and Example \ref{ex:2}, we denote by $P = \begin{bmatrix}
    0&1\\
    1&0
\end{bmatrix}$ the $2 \times 2$ permutation matrix other than the identity

\begin{example}\label{ex:1}
We have already seen that, if $A$ is a multiple of the identity, then $f_A(Q)$ can have uncountably many local minima on $\mathcal{F}_2$. For a more interesting example of a non-normal input that is ``bad" in this sense, take $\displaystyle A = \begin{bmatrix}
    0&1\\
    0&0
\end{bmatrix} \in \C^{2 \times 2}$.
In this case \( \dim_\R \mathcal{F}_2 = 2 \) and, as representative of each equivalence class in the flag manifold, we can select a unitary matrix with real nonnegative diagonal elements. These matrices can be parametrized by two angles $\alpha \in [-\pi/2,\pi/2],\beta \in [0,2\pi]$ as
\begin{equation}\label{eq:Qab}
     Q(\alpha,\beta) = \begin{bmatrix}
    \cos(\alpha) & \sin(\alpha) e^{i \beta}\\
    -\sin(\alpha) e^{-i\beta} & \cos(\alpha)
\end{bmatrix}
\end{equation}
we obtain $f_A(\alpha,\beta)=\cos(\alpha)^4 + \sin(\alpha)^4$. The local minimizers are therefore only determined by $\alpha$, with $\beta$ free. Equivalently, in the language of Theorem \ref{thm:sard}, $A \in \mathcal{J}$. In particular, the global minima occur at $\alpha=\pm \frac{\pi}{4}$ corresponding to
\[   Q = \frac{1}{\sqrt{2}} \begin{bmatrix}
    1 & \pm e^{i \beta}\\
    \mp e^{-i \beta} & 1
\end{bmatrix}. \]
We note that this result also implies that the normal matrices nearest to $A$ are 
\[ N(\beta) = \frac12 \begin{bmatrix}
    0 & 1\\
    e^{-2i\beta} & 0
\end{bmatrix}, \qquad \mathrm{with} \ \|A-N(\beta)\|_F^2 = \frac12.\]
\end{example}

Note that one can always rotate $A \in \Cnn$ by unitary similarity to make it upper triangular \cite[Theorem 2.3.1]{HoJo}, shift it by a multiple of the identity to make it traceless, and scale it to make its top left element equal to either $0$ or $1$. This procedure constructs a canonical input $C:=k(Q^*AQ-cI) \in \Cnn$, where $Q \in U(n), c \in \C$ and $0 \neq k \in \C$. It is clear that normality is preserved under all these transformations. Moreover,
\[ \| A - N \|_F = \frac{1}{|k|} \| C - k(Q^*NQ-c I)\|_F,\]
and 
\[ f_C(U) =  |k|^2 f_A(QU)   \quad \forall \ U \in U(n).             \]
By these observations, there is a bijection between the sets of the normal matrices nearest to $A$ and of those nearest to $C$, and the local minimizers of $f_A$ on $\mathcal{F
}_n$ (and on $U(n)$) are also mapped by a simple bijection. When $n=2$, this means that the possible cases are parametrized by just one complex parameter $z$. Indeed, if $A$ has two equal eigenvalues, then either $C=0$, for which $f_C(Q)=0$ is constant and the nearest normal matrix is obviously $C$ itself, or  $C=\begin{bmatrix}
    0&1\\
    0&0
\end{bmatrix}$ (discussed in Example \ref{ex:1}). If instead $A$ has distinct eigenvalues then $C=\begin{bmatrix}
    1 & z\\
    0 & -1
\end{bmatrix}, z \in \C$. We analyze this last remaining case in Example \ref{ex:2} below.
\begin{example}\label{ex:2}
    Let 
    \[  A = \begin{bmatrix}
        1 & \rho e^{i \theta}\\
        0 & -1
    \end{bmatrix} \in \Cnn, \qquad \rho \geq 0, \  -\pi < \theta \leq \pi;   \]
    parametrizing $Q$ as in \eqref{eq:Qab} we obtain
    \[  f_A(\alpha,\beta)= \frac{\rho^2}{2} + \frac12 \left| 2 \sin(2\alpha) + \rho e^{i (\beta-\theta)} \cos(2 \alpha) \right|^2. \] 
    
    By computing the gradient and Hessian (we omit the details), one can see that the local minimizers on $\mathcal{F}_2$ are precisely

    \[ \left[Q\left(-\frac12\arctan\frac{\rho}{2},\theta \right)\right] \quad \text{and} \quad \left[Q\left(-\frac12\arctan\frac{\rho}{2},\theta \right) P \right]; \]
    they are also global minimizers.
    
   If instead $\rho=0$, the local (and global) minimizers  are given by  $\alpha \in \{ 0, \pm \frac{\pi}{2} \}$ and $\beta$ is free; however, in this case the freedom in $\beta$ does not correspond to infinitely many local minimizers on the flag manifold. Indeed, substituting in \eqref{eq:Qab}, we obtain the points $[I], \left[ P\right] \in \mathcal{F}_2$.
\end{example}

Our analysis above reveals that for $n=2$ the set of inputs $A$ for which $f_A$ has infinitely many local minimizers on $\mathcal{F}_n$ is
\[
   \{ A = c I,  \ c \in \C \} \cup \{A \ \mathrm{is} \ \mathrm{defective} \} \subsetneq \C^{2 \times 2}.   
\]

Another interesting feature of the $n=2$ case is that every local minimizer of $f_A(Q)$ is also a global minimizer. This can be seen by analyzing the Hessian of the stationary points computed above, and verifying that they are either global minima, saddle points, or maxima; we omit the details. This property is extremely attractive numerically, but unfortunately it does not generalize  to $n>2$.

\begin{example}\label{ex:3}
    Let
    \[ A=\begin{bmatrix}
        5-10i & 6+8i & 0\\
        10 & 10i-5 &3-4i\\
        0 & 5i & -5i
    \end{bmatrix} \in \C^{3 \times 3}.   \]
    Substituting this particular $A$ and $Q=I$ in \eqref{eq:fgrad}, it can be checked that the identity matrix is a stationary point of $f_A$. It is a marginally more laborious exercise to compute the Hessian of $f_A$ at the identity by using \eqref{eq:fhess_quotient}. We omit the tedious derivation and just mention that the outcome is a positive definite Hessian, proving that $I$ is a local minimum for $f_A$, corresponding of course to the normal matrix $\diag(A)$. By running our code (see Section \ref{sec:numexp}), we found the point (normalized to have positive diagonal)
    \[  Q_0 \approx \begin{bmatrix}
        0.6391   & 0.5163 + 0.2581i &  0.3594 - 0.3594i\\
  -0.0731 + 0.0366i &  0.7081  & -0.2218 + 0.6653i\\
  -0.5408 - 0.5408i &  0.1286 + 0.3858i  & 0.4998 
    \end{bmatrix} . \]
    Applying Theorem \ref{prop:equivalent}, we associate $Q_0$ with the normal matrix
    \[ N_0 \approx \begin{bmatrix}
        0.9386 - 1.9684i  & 5.0807 + 6.7195i &  1.8143 + 3.8169i\\
   8.4240 - 0.0329i & -0.7206 + 1.0916i  & 4.1969 - 5.5201i\\
  -3.8169 + 1.8143i & -0.0455 + 6.9342i & -0.2179 - 4.1231i
    \end{bmatrix},\]
    and $15.8114 \approx \sqrt{250} = \|A-\diag(A)\|_F =\sqrt{f_A(I)} > \sqrt{f_A(Q_0)}=\|A-N_0\|_F \approx 15.0887$.
\end{example}

\section{The real case}\label{sec:real}
Suppose $A \in \Rnn$, and that we are interested in the \emph{real} normal matrix $N \in \Rnn$ nearest to $A$. It is an open question \cite{Higham} whether $N$ is also always a nearest matrix to $A$ among all complex normal matrices\footnote{Empirically, we never observed a real input for which our real algorithm found a better approximation than our complex algorithm. On the other hand, we observed real inputs for which the best normal matrix found by the complex algorithm was nonreal and strictly closer to the input than the best matrix found by the real algorithm. These observations do not settle the open question, since neither algorithm is guaranteed to compute a global minimum.}. 
A na\"{i}ve approach would be to run our complex algorithm, or any iterative Riemannian optimization algorithm, with a real orthogonal matrix $Q \in O(n)$ as a starting point. Then, the algorithm would only employ real arithmetic, forcing every subsequent iterate to stay in $O(n)$. The problem with this approach is that it forces the eigenvalues of the corresponding normal matrices to stay real, and hence it must ultimately compute a symmetric matrix. As a consequence, the output of such a method cannot possibly be any closer to $A$ than $\frac{A+A^T}{2}$, which can of course be computed effortlessly, making it pointless to utilize the complex algorithm in this way.

Inspired by \cite[Section 4]{NP}, we now describe a more sophisticated real version of the algorithm. Proposition \ref{prop:realnormal} is known \cite{Chu,Gabriel79}, but we still include a short proof to keep the paper self-contained.

\begin{proposition}\label{prop:realnormal}
    A real matrix $A \in \Rnn$ is normal if and only if there exists an orthogonal matrix $Q$ such that $Q^T A Q =D$, where $D$ is a block diagonal real matrix whose diagonal blocks are either $1 \times 1$ matrices or $2 \times 2$ matrices of the form
    $ \begin{bmatrix}
        a & b\\
        -b & a
    \end{bmatrix}$.
\end{proposition}
\begin{proof}
    Note the field isomorphism
    \[ \varphi: \C \mapsto \varphi(\C) \subset \R^{2 \times 2}, \quad \varphi(a+ib)=\begin{bmatrix}
        a&b\\
        -b&a
    \end{bmatrix}. \]
   In particular all $2 \times 2$ matrices of that form commute. This immediately implies that, if $A$ satisfies the stated equation, then $AA^T=A^TA$.

    Suppose now that $A$ is normal, and write its spectral decomposition $A=U \Lambda U^*$, where without loss of generality the eigenvalues in $\Lambda$ are ordered so that the real eigenvalues come first, and the complex conjugate pairs are grouped together. If $A$ has $n-2m$ real eigenvalues, setting $\phi(j)=n-2m+2j-1$ we can then write 
    \[ A=\sum_{i=1}^{n-2m} \lambda_i u_i u_i^* +  \sum_{j=1}^m \left( \lambda_{\phi(j)} u_{\phi(j)} u_{\phi(j)}^* + \overline{\lambda}_{\phi(j)} u_{\phi(j)+1} u_{\phi(j)+1}^* \right).\]
    Moreover, we can assume without loss of generality that $u_1, \dots, u_{n-2m} \in \R^n$ and that, for all $j=1,\dots,m$, $u_{\phi(j)+1}=\overline{u_{\phi(j)}}$. It now suffices to observe that, if $\lambda_{\phi(j)}=:a+ib$ and $u_{\phi(j)}=x+iy$ where $a,b \in \R$ and $x,y \in \R^n$, then
\[  \begin{bmatrix}
    x+iy & x-iy
\end{bmatrix}  \begin{bmatrix}
    a+ib&0\\
    0&a-ib
\end{bmatrix}\begin{bmatrix}
    x^T-i y^T\\
    x^T + i y^T
\end{bmatrix} = \begin{bmatrix}
    \sqrt{2} \cdot x & \sqrt{2} \cdot  y
\end{bmatrix} \begin{bmatrix}
    a&b\\
    -b &a
\end{bmatrix} \begin{bmatrix}
    \sqrt{2}  \cdot x^T\\
    \sqrt{2} \cdot y^T
\end{bmatrix}  \]
and
\[ \begin{bmatrix}
  x+iy & x-iy  
\end{bmatrix} \left(\begin{bmatrix}
    1&-i\\
    1&i
\end{bmatrix}/\sqrt{2}\right)  = \begin{bmatrix}
   \sqrt{2} \cdot x & \sqrt{2} \cdot y
\end{bmatrix}  \]
so that unitarity is preserved.
\end{proof}

Proposition \ref{prop:realnormal} implies that
minimizing $\|A-N\|_F^2$ over real normal matrices $N$ is equivalent to 
minimizing, over the manifold of $n \times n$ orthogonal matrices $O(n)$, the function 
\begin{equation}
    \label{def:freal} 
    f(Q) = \| \mathcal{N}_\R(Q^T A Q)\|_F^2
\end{equation}
where $\mathcal{N}_\R(M)=M-\Pi(M)$ and $\Pi(M)$ projects onto a block diagonal matrix having $2 \times 2$ blocks (plus one $1 \times 1$ block at the bottom if $n$ is odd). The diagonal blocks of $\Pi(M)$ are defined from the corresponding blocks of $M$ as follows: If $B \in \R$ is a $1 \times 1$ block,  then the projected block is $\Pi(B)=B$ itself, while if the block is $B=\begin{bmatrix}
    M_{11} & M_{12}\\
    M_{21} & M_{22}
\end{bmatrix} \in \R^{2 \times 2}$,  then the projected block is
\[ \Pi(B)= \begin{cases}
\frac12 \begin{bmatrix}
     M_{11}+M_{22} & M_{12}-M_{21}\\
     M_{21}-M_{12} & M_{11}+M_{22}
 \end{bmatrix}   & \mathrm{if} \ |M_{11}-M_{22}|<|M_{12}-M_{21}|  \ ;\\
  \begin{bmatrix}
     M_{11}&0\\
     0&M_{22}
 \end{bmatrix}  & \mathrm{otherwise}.
\end{cases}    \]
For example,
\[  \Pi\left( \begin{bmatrix}
    1&2&24&23&21\\
    3&4&8&11&13\\
    5&0&12&10&15\\
    6&22&16&14&17\\
    7&20&19&18&9
\end{bmatrix} \right) =  \begin{bmatrix}
    1&0&0&0&0\\
    0&4&0&0&0\\
    0&0&13&-3&0\\
    0&0&3&13&0\\
    0&0&0&0&9
\end{bmatrix}.  \]
\begin{remark}
    Reducing the dimension of the search space by considering a quotient manifold of $O(n)$, similarly to what we did in the complex case by considering the flag manifold, is not viable for the real algorithm. The reason is that the correct choice of the quotient depends on the number of real eigenvalues of a nearest real normal matrix, and this information is not a priori available. Such refinement is, however, possible in the scenario considered in \cite{Chu}, where the target spectrum \emph{is} assumed to be known.
\end{remark}
\subsection{Real gradients and Hessians}

Our opening move is to reformulate Lemma \ref{lm:costgradient} and Corollary \ref{cor:Egradf} in the real setting. We omit the analogous proofs of Lemma \ref{lm:costgradientreal} and Corollary \ref{lm:realgradient}.

\begin{lemma}
    \label{lm:costgradientreal}
    Let $A \in \Rnn$ and let $h\colon O(n)\to \R$ be a function of the form $h(Q) = \|\mathcal{P}(Q^TAQ)\|^2_F$, where $\mathcal{P}$ is a linear orthogonal (in the Frobenius inner product) projection defined in a neighborhood of $Q\in O(n).$ Then, the Euclidean gradient of $h$ at $Q$ is
         \begin{equation*}
             \nabla h(Q) = 2(A^TQS + AQS^T),
         \end{equation*}
         and its Riemannian gradient in $Q$ equals to 
         \begin{equation*}
             \grad h(Q) = 2Q\skeww([B^T, S]),
         \end{equation*}
         where $B = Q^TAQ,\, S = \mathcal{P}(B),$ and $\skeww(X) = \frac{1}{2}(X - X^T).$
\end{lemma}

\begin{corollary}
\label{lm:realgradient}
Let $f_A: O(n) \to \R$ be defined as in \eqref{def:freal}, and define $B=Q^T A Q$ and $S=\mathcal{N}_\R(B)$. Then, 
    \begin{equation*}
        \label{eq:fgradreal}
        \nabla f(Q) = 2 Q(B^TS + BS^T), \qquad \grad f(Q) = 2Q\skeww([B^T, S]),
    \end{equation*}
    for every $Q\in O(n)$ in the open set on which the branch defining \(\mathcal N_\R \) is locally constant.
\end{corollary}

\begin{remark}
  The statement of Corollary \ref{lm:realgradient} correctly hints at the lack of global differentiability of the objective function $f_A$ in \eqref{def:freal}. However, $f_A$ is locally Lipschitz, and hence it is almost everywhere differentiable by Rademacher's Theorem.  To refine this observation, note that the operator $\mathcal{N}_\R$ acts as a piecewise linear projection onto finitely many (more precisely, $2^{\lfloor n/2 \rfloor}$) possible subspaces of $\Rnn$. The choice between those subspaces is determined by the sign of $|M_{11}-M_{22}| - |M_{12}-M_{21}|$, evaluated on the relevant $2 \times 2$ principal submatrices $\displaystyle M=\begin{bmatrix}
      M_{11} & M_{12}\\
      M_{21} & M_{22}
  \end{bmatrix}$ of $Q^TAQ$. Within the open and dense (for generic $A$) subset of $O(n)$ where these signs are nonzero, $\mathcal{N}_\R$ behaves as an orthogonal projection onto a fixed subspace, and thus $f_A$ is smooth.
\end{remark}

We now deduce a formula for the Riemannian Hessian in the real case. Again, the proof of Corollary \ref{lm:realhess} is omitted because it is analogous to that of Corollary \ref{cor:Rhessian}.

 \begin{corollary}
      \label{lm:realhess}
            Let $f_A$ be as in \eqref{def:freal} and let $Q\in O(n)$ be in the open set on which the branch defining \(\mathcal N_\R \) is locally constant. Then, for every $Z=Q \Omega \in \mathcal{T}_Q O(n)$, where $\Omega=-\Omega^T$, the Riemannian Hessian of $f_A$ on $O(n)$ is 
            \begin{equation*}
                \hess f(Q)[Q\Omega] = 2Q\skeww(\Omega\skeww([B^T, S]) + [[B, \Omega]^T, S] + [B^T, \mathcal{N}_{\R,B}([B, \Omega])]),
            \end{equation*}
            where $\skeww(X) = \frac{1}{2}(X - X^T)$ and $\mathcal{N}_{\R,B}$ denotes the branch of the piecewise linear projection $\mathcal{N}_\R$ selected according to $B$. 
 \end{corollary}

\section{Algorithmic aspects}\label{sec:numexp}

\subsection{Implementation details}\label{sec:implement}

\subsubsection{Choice of optimization algorithm}

 In our practical implementation, we use the Riemannian trust-region method in Manopt \cite{Manopt}. We have compared its performance against other solvers available in the Manopt toolbox, including adaptive regularization by cubics (ARC) and first-order methods such as conjugate gradient. The trust-region algorithm outperformed these alternatives on both Gaussian random matrices and perturbed normal matrices.

\subsubsection{Starting point}
The algorithms described in Section~\ref{sec:riemann} converge in practice to a local minimizer of the objective function $  f_A  $. Hence, the robustness of the method and the likelihood of reaching the basin of attraction of a global minimizer are improved by selecting an initial point that is already close to such a minimizer.
One suitable choice is provided by the Schur decomposition of the matrix $  A  $ \cite[Theorem~2.3.1]{HoJo}. This yields a unitary (respectively, orthogonal) matrix $  Q_A  $ such that $  Q_A^* A Q_A  $ is upper-triangular (respectively, quasi-upper-triangular). Numerical experiments on perturbed normal matrices show that initializing with the Schur factor $  Q_A  $ approximately halves the average runtime relative to both a random starting point and the identity matrix.

However, while the Schur initialization outperforms the random initialization for generic matrices, it might still fail to reach the global minimum's basin of attraction. Most notably, for Frank matrices which are used as an example in \cite{Ruhe}, we have observed that the Schur initialization typically fails to find a global minimum.

\subsubsection{Choice of manifold}
As discussed in Section \ref{sec:riemann}, our algorithm can be implemented indifferently on $U(n)$ or on $\mathcal{F}_n$. The lower dimension of $\mathcal{F}_n$ is potentially advantageous; moreover, over $U(n)$, the Hessian of the objective function $f_A$ is only positive semidefinite (as opposed to definite) at a local minimum, because by Proposition \ref{prop:symmetries} $f_A$ is constant along a submanifold of dimension $n$. Points in favor of $U(n)$ include the fact the flag manifold is not native in Manopt, and a manual implementation of its geometry may cause overheads; moreover, further algebraic computations are needed to project quantities onto the horizontal space (see Corollary \ref{cor:Rhessian}).

A priori, it is unclear which of these aspects is the most relevant. In Table \ref{tab:UnvsFn}, we evaluated numerically, in terms of both number of iterations and running time, the performances of two implementations of our method that are identical except for the choice of the underlying manifold. We tested them on three types of input: Realizations of Gaussian random matrices, normal matrices of norm $1$ plus a perturbation of norm $10^{-3}$, and matrices of the form $V J_n(0) V^{-1}$, where $V$ is a similarity matrix having condition number $10$ and $J_n(0)$ is a nilpotent Jordan block. This experiment was essentially inconclusive. We observed both specific inputs where $\mathcal{F}_n$ outperformed  $U(n)$ and vice versa. On average, no clear trend emerged. The peak average performance gap (in favor of $\mathcal{F}_n$) was $\approx 7.4 \%$, for $90 \times 90$ Gaussian matrices. In our broader numerical experiments below, we default to using $U(n)$, because its implementation in Manopt is more straightforward. However, we cannot exclude that $\mathcal{F}_n$ may be preferable for specific instances of the problem, and for this reason our code leaves to the user the option to launch the implementation on the flag manifold.

\begin{table}[htpb]
\centering
\resizebox{\textwidth}{!}{
\begin{tabular}{c cccc cccc cccc}
\toprule
 & \multicolumn{4}{c}{Gaussian matrices} & \multicolumn{4}{c}{Perturbed normal} & \multicolumn{4}{c}{Nilpotent matrices} \\
\cmidrule(lr){2-5} \cmidrule(lr){6-9} \cmidrule(lr){10-13}
 & \multicolumn{2}{c}{Iterations} & \multicolumn{2}{c}{Time (s)} & \multicolumn{2}{c}{Iterations} & \multicolumn{2}{c}{Time (s)} & \multicolumn{2}{c}{Iterations} & \multicolumn{2}{c}{Time (s)} \\
\cmidrule(lr){2-3} \cmidrule(lr){4-5} \cmidrule(lr){6-7} \cmidrule(lr){8-9} \cmidrule(lr){10-11} \cmidrule(lr){12-13}
$n$ & $U(n)$ & $\mathcal{F}_n$ & $U(n)$ & $\mathcal{F}_n$ & $U(n)$ & $\mathcal{F}_n$ & $U(n)$ & $\mathcal{F}_n$ & $U(n)$ & $\mathcal{F}_n$ & $U(n)$ & $\mathcal{F}_n$ \\
\midrule
$30$ & 504.6 & 509.3 & 0.94 & 0.94 & 177.8 & 178.0 & 0.33 & 0.32 & 782.1 & 793.9 & 1.50 & 1.52 \\
$60$ & 1209.8 & 1233.5 & 3.65 & 3.70 & 612.5 & 618.5 & 1.84 & 1.84 & 2147.3 & 2091.4 & 6.79 & 6.76 \\
$90$ & 2013.9 & 1874.7 & 9.19 & 8.57 & 1156.2 & 1156.4 & 5.18 & 5.23 & 2966.5 & 2993.0 & 15.01 & 15.59 \\
\bottomrule
\end{tabular}
}
\caption{Performance comparison between $U(n)$ and $\mathcal{F}_n$. Results are averaged over $100$ randomly generated inputs for each size and type. Initialization via Schur decomposition.}
\label{tab:UnvsFn}
\end{table}

\subsection{Numerical experiments}\label{sec:numexpproper}

We are not aware of any established benchmark test set for the nearest normal matrix. However, some analytic examples have appeared in \cite{Barl,GS,Ruhe}. We report in Table \ref{tab:prevex} a comparison of our algorithm with the results given in \cite{GS}. Remarkably, even for small-size inputs, our method strictly improved the state of the art for three out of four input matrices.

\begin{table}[htpb]
\centering
\resizebox{0.7\textwidth}{!}{%
\begin{tabular}{c cccc cc}
\toprule
 & \multicolumn{4}{c}{$\|A - A_*\|_F$} & \multicolumn{2}{c}{Time, s} \\
\cmidrule(lr){2-5} \cmidrule(lr){6-7}
$A$ & Ruhe & GS & NZC & NZR & NZC & NZR \\
\midrule
$J_7$, \cite{Ruhe}& \textbf{0.926} & N/A & \textbf{0.926} & \textbf{0.926} & 0.02 & 0.01\\
$F_{12}$, \cite{Ruhe} & 22.586 & N/A & 22.582 & \textbf{22.581} & 0.19 & 0.12\\
Ex. 1, \cite{GS} & 0.959$^\dagger$ & 1.067$^\ddagger$ & \textbf{0.842} & \textbf{0.842} & 0.01 & 0.004\\
Ex. 2, \cite{Barl} & N/A & 22.267 & \textbf{22.263} & N/A & 0.045 & N/A\\
\bottomrule
\end{tabular}%
}

\smallskip
\begin{minipage}{0.7\textwidth}
{\footnotesize
$^\dagger$ According to \cite{GS}; the example is not included in \cite{Ruhe}.\\
$^\ddagger$ \cite{GS} considers only the nearest real normal matrix problem.
}
\end{minipage}

\caption{Comparison between methods. $A_*$ denotes the computed approximation; distances are in boldface when they are the best within a relative threshold of $10^{-6}$. GS denotes the method from \cite{GS}, Ruhe denotes the method from \cite{Ruhe}, whereas NZC and NZR refer to the complex and the real versions of our algorithm, respectively, run with a randomly generated starting point.}
\label{tab:prevex}
\end{table}

\subsubsection{Empirical frequency of nonglobal minima}

Example \ref{ex:3} shows that, for $n \geq 3$, the function $f_A$ might have nonglobal local minima, implying that our algorithm cannot guarantee to find the exact distance. We have empirically analyzed the occurrence of this phenomenon for small $n$ (for reasons of efficiency; this experiment is computationally very demanding given its statistical nature). For each size, we generated $500$ random Gaussian complex inputs and ran the algorithm from $100$ random starting points per input. We report the outcome in Table \ref{tab:mult} below: It frequently happens that the algorithm converges always to the same local minimum (quite possibly the global one); moreover, when there are multiple minima, the probability of converging to the best one from a random point is high.

\begin{table}[htpb]
\centering
\begin{tabular}{c cc}
\toprule
$n$ & $\%$ one minimum &
$\%$ best when multiple minima \\
\midrule
$3$  & $99\%$ & $75\%$ \\
$5$  & $97\%$ & $64\%$ \\
$10$ & $90\%$ & $80\%$ \\
\bottomrule
\end{tabular}
\caption{Empirical occurrence of multiple local minima for Gaussian inputs. From left to right, the columns report the input size, the fraction of inputs for which the algorithm always converged to the same local minimizer, and the probability that the algorithm converged to the best observed minimizer (conditioned on having observed at least two distinct minimizers).}
\label{tab:mult}
\end{table}

\subsubsection{Empirical complexity}

We tested both algorithms on randomly generated matrices of sizes varying from 4 to 256. For each size, 300 inputs were tested, generated as random normal matrices with the addition of 1\% Gaussian noise. Figure \ref{fig:generaltiming} presents the results for the complex version of the algorithm and figure \ref{fig:realtiming} presents the results for the real version. The tests were run on a Macbook Air machine with an Apple M2 chip.
\begin{figure}[htpb]
    \centering
    \includegraphics[width=0.8\linewidth]{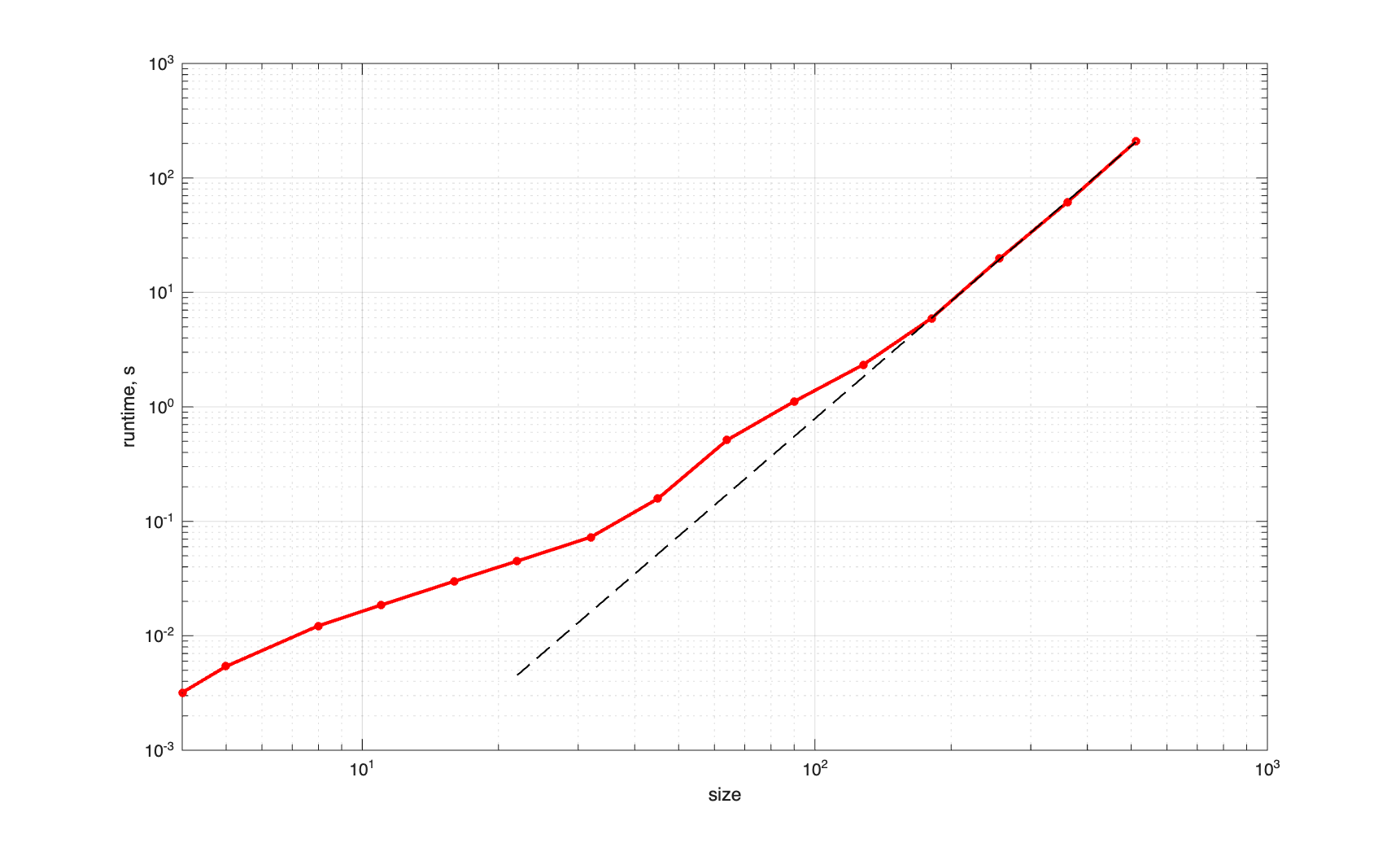}
    \caption{Median runtime for the complex algorithm. The dashed line represents a polynomial fit of the tail, numerically computed as $O(n^{3.4}).$}
    \label{fig:generaltiming}
\end{figure}
\begin{figure}[htpb]
    \centering
    \includegraphics[width=0.8\linewidth]{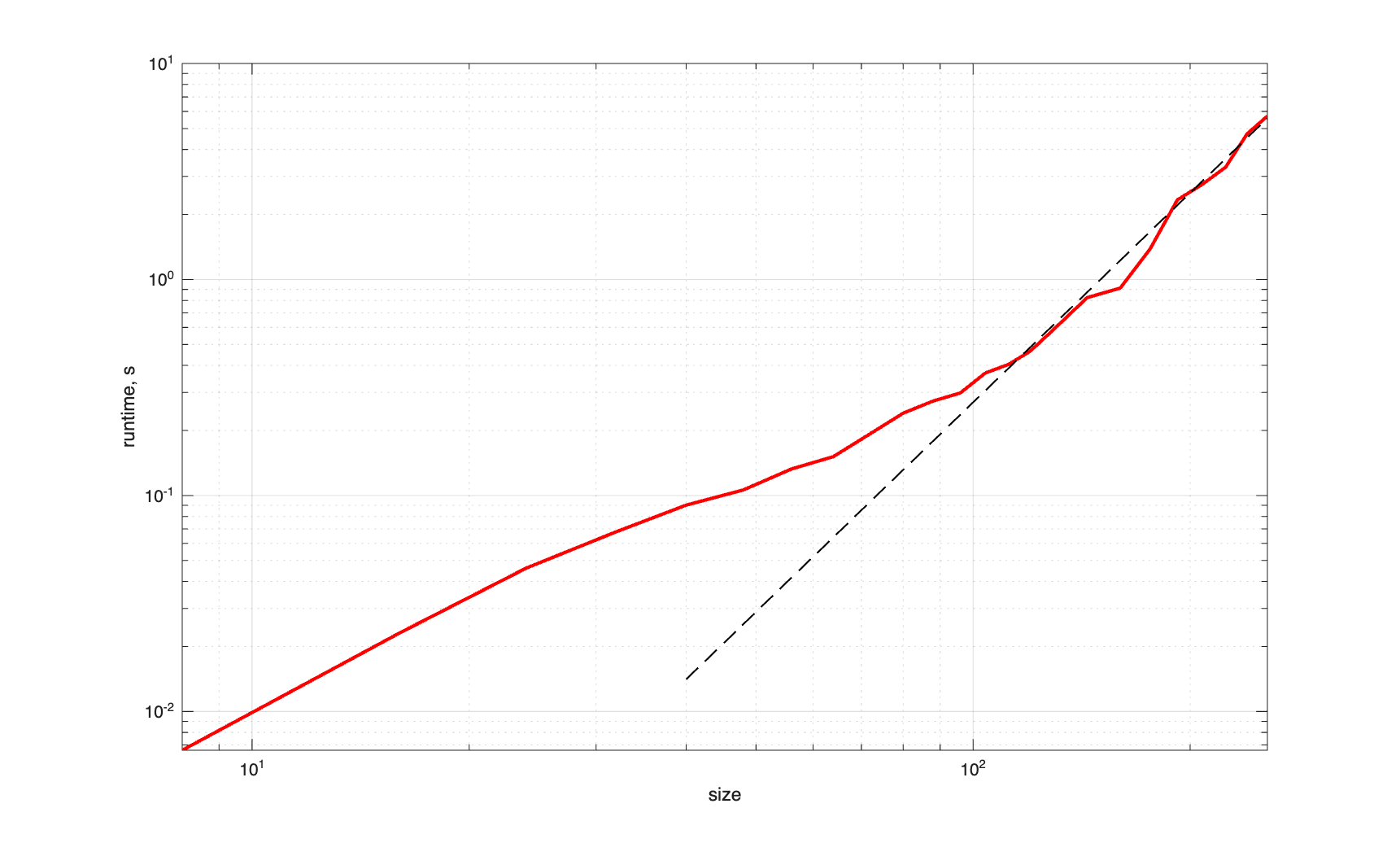}
    \caption{Median runtime for the real algorithm. The dashed line represents a polynomial fit of the tail, numerically computed as $O(n^{3.25}).$}
    \label{fig:realtiming}
\end{figure}

\section{Conclusions}\label{sec:conclusions}

We revisited the classical nearest normal matrix problem using modern Riemannian optimization, deriving a practical method that substantially improves on classical algorithms and can handle considerably larger matrices. We also developed new theory, and in particular we proved that, for generic inputs, the local minimizers on the flag manifold are isolated and finitely many, implying an analogous result on the original nearest normal matrix problem. A practical question for future research concerns more reliable initialization strategies, especially in the real case.

\section*{Acknowledgements} We are grateful to Mark Embree, whose comments during the Householder Symposium XXII inspired this paper, to Federico Poloni for comments on the manuscript, and to Kyle Bierly for bringing \cite{Kyle} to our attention.


\begin{thebibliography}{9}

\bibitem{ABG07} P.-A. Absil, C. Baker and K. Gallivan.
\newblock {\em Trust-region methods on Riemannian manifolds}.
\newblock Found. Comput. Math. 7(3), 303--330, 2007.

\bibitem{AMS08} P.-A. Absil, R. Mahony and R. Sepulchre. 
\newblock {\em Optimization Algorithms on Matrix Manifolds}. 
\newblock Princeton University Press, 2008.

\bibitem{Barl}  F. Barl. \emph{Higher-Rank Numerical Range of Almost Normal Matrices},  MSc Thesis, 2014.

\bibitem{Kyle} K. Bierly, \emph{The normal Procrustes problem: A Riemannian optimization approach}. Preprint, 2026.

% \bibitem{BCR} J. Bochnak, M. Coste and M.-F. Roy.
% \newblock {\em Real Algebraic Geometry}.
% \newblock Springer-Verlag, Berlin, 1998.

\bibitem{Boumal} N. Boumal. {\em An introduction to optimization on smooth manifolds}. Cambridge University Press,
 2023.

 \bibitem{Manopt} N. Boumal, B. Mishra, P.-A. Absil and R. Sepulchre. {\em Manopt, a MATLAB toolbox for optimization
on manifolds}. Journal of Machine Learning Research 15(42):1455–1459, 2014


\bibitem{Causey} R. L. Causey, \emph{On Closest Normal Matrices}, Ph.D. Thesis, Department of Computer
Science, Stanford University, 1964.


\bibitem{Chu} M. T. Chu, \emph{Least squares approximation by real normal matrices with specified spectrum}, SIAM J. Matrix Anal. Appl. 12(1):115-127, 1991.

% \bibitem{CLO} D. A. Cox, J. Little and D. O'Shea.
% \newblock {\em Ideals, Varieties, and Algorithms}. 4th ed.
% \newblock Springer, Cham, 2015.

\bibitem{DK83} R. W. Daniel and B. Kouvaritakis, \emph{The choice and use of normal matrix approximations to transfer-function matrices of multivariable control systems}, Int. J.
Control 37:1121–1133, 1983.

\bibitem{DK84} R. W. Daniel and B. Kouvaritakis, \emph{Analysis and design of linear multivariable
feedback systems in the presence of additive perturbations}, Int. J.
Control 39:551–580, 1984.

\bibitem{DNN} F. Dopico, V. Noferini and L. Nyman, \emph{A Riemannian optimization method to compute the nearest singular pencil}, SIAM J. Matrix Anal.  Appl. 45(4):2007-2038, 2024.

\bibitem{EI}L. Elsner and K. D. Ikramov, \emph{Normal matrices: an update}, Linear Algebra Appl. 285:291-303, 1998.

\bibitem{Erdos} P. Erd\H{o}s, \emph{Some remarks on the measurability of certain sets}, Bull. Amer. Math. Soc.
51(11):728–731, 1945.

\bibitem{Fillmore}
P. A. Fillmore, \emph{On similarity and the diagonal of a matrix},
Am. Math. Mon. 76:167-169, 1969.

\bibitem{Gabriel79}
R. Gabriel, \emph{Matrizen mit maximaler Diagonale bien unit\"{a}re Similarit\"{a}t}, J. Reine Angew. Math. 307/308:31-52, 1979.

\bibitem{Gabriel} R. Gabriel, \emph{The normal $\Delta H$-matrices with connection to some Jacobi-like methods}, Linear Algebra Appl. 91:181--194, 1987.

\bibitem{GV} G. H. Golub and C. F. Van Loan. \emph{Matrix Computations} (4th ed.). Johns Hopkins University Press, 2013.

\bibitem{GJSW}
R. Grone, C. R. Johnson, E. M. Sa and H. Wolkowicz, \emph{Normal matrices}, Linear Algebra
Appl. 87:213–225, 1987.


\bibitem{GS} N. Guglielmi and C. Scalone, \emph{Computing the closest real normal matrix
and normal completion}, Adv. Comput. Math. 45:2867–2891, 2019.

\bibitem{Higham} N. J. Higham, \emph{Matrix nearness problems and applications}. In M. J. C. Gover and S. Barnett,
editors, Applications of Matrix Theory, pages 1–27. Oxford University Press, 1989.

\bibitem{HoJo}
R. Horn and C. R. Johnson. \emph{Matrix Analysis} (2nd ed.). Cambridge University Press, 2013.

\bibitem{Lee} J. M. Lee. 
\newblock {\em Introduction to Smooth Manifolds} (2nd ed.).
\newblock Springer, New York, 2012.


\bibitem{MV} C. Moler and C. Van Loan, \emph{Nineteen dubious ways to
compute the exponential of a
matrix, twenty-five years later}, SIAM Rev. 45(1):3--49, 2003.

\bibitem{NN} V. Noferini and L. Nyman, \emph{Finding the nearest $\Omega$-stable pencil with Riemannian optimization}, Numer. Algo. 102:569–592, 2026.

\bibitem{NP} V. Noferini and F. Poloni, \emph{Nearest $\Omega$-stable matrix via Riemannian optimization}, Numer. Math. 148:817--851, 2021.

\bibitem{NPacta} V. Noferini and F. Poloni, \emph{Matrix Nearness Problems Revisited}, In preparation, 2026.

\bibitem{NPR} S. Noschese, L. Pasquini and L. Reichel, \emph{The structured distance to normality of an irreducible real tridiagonal matrix}, Electronic Transactions on Numerical Analysis 28:65--77, 2007.

\bibitem{Ruhe} A. Ruhe, \emph{Closest normal matrix finally found!}, BIT Numer. Math. 27:585--598, 1987.

\bibitem{SN} J. J. Sakurai and J. Napolitano.
\newblock {\em Modern Quantum Mechanics} (3rd ed.).
\newblock Cambridge University Press, 2020.

\bibitem{TE} L. N. Trefethen and M. Embree, \emph{Spectra and Pseudospectra: The Behavior of Nonnormal Matrices and Operators}, Princeton University Press, 2005.

\end{thebibliography}
\end{document}